\documentclass[12pt]{amsart}

\usepackage{enumerate}
\usepackage{enumitem}
\usepackage{graphicx}

\usepackage{geometry}
\newtheorem{theorem}{Theorem}[section]
\newtheorem{lemma}[theorem]{Lemma}

\theoremstyle{definition}

\theoremstyle{remark}

\numberwithin{equation}{section}

\newcommand{\norm}[1]{\left\lVert#1\right\rVert}

\DeclareMathOperator{\buc}{BUC}
\DeclareMathOperator{\lip}{LIP}
\DeclareMathOperator{\AC}{AC}

\begin{document}

\title[Optimal convergence rate]{Optimal convergence rate for homogenization of convex Hamilton-Jacobi equations in the periodic spatial-temporal environment}


%    author one information
\author{Hoang Nguyen-Tien}
\address{}
\curraddr{}
\email{}
\thanks{This work is partly supported by grant U2022-01 at the VNU-HCM University of Science}

%    author two information
% \author{}
% \address{}
% \curraddr{}
% \email{}
% \thanks{}

%    \subjclass is required.

\date{\today}

\dedicatory{}

%    "Communicated by" -- provide editor's name; required.
% \commby{}

%    Abstract is required.
\begin{abstract}
  We study the optimal convergence rate for homogenization problem of convex Hamilton-Jacobi equations when the Hamitonian is periodic with respect to spatial and time variables, and notably time-dependent. We prove a result similar to that of \cite{hungyu}, which means the optimal convergence rate is also $O(\varepsilon)$.
\end{abstract}

\maketitle

%    Section headings
\section{Introduction}
\subsection{Settings}

For each $\varepsilon>0$, let $u^\varepsilon\in C(\mathbb R^n\times [0,\infty))$ be an unique viscosity solution of the following Cauchy problem
\begin{equation}\label{hj-cauchy-timedep}
    \begin{cases}
        u_t^\varepsilon+H\left(\dfrac{x}{\varepsilon},\dfrac{t}{\varepsilon},Du^\varepsilon\right)&=0\quad\quad \ \ \text{in $\mathbb R^n\times (0,\infty)$},\\
        u^\varepsilon(x,0)&=g(x)\quad \ \text{on $\mathbb R^n$}.
    \end{cases}
\end{equation}

The Hamiltonian $H:\mathbb R^n\times \mathbb R\times \mathbb R^n\rightarrow \mathbb R$ is a given continuous function satisfying:
\begin{enumerate}[label=\roman*)]
    \item[(P1)] for each $p\in \mathbb R^n$, the map $(x,t)\mapsto H(x,t,p)$ is $\mathbb Z^{n+1}-$periodic, which means
        $$H(x,t,p)=H(x+u,t+v,p)\quad \forall (u,v)\in \mathbb Z^n\times \mathbb Z;$$
    \item[(P2)] there exist some positive constants $\alpha_0,\beta_0, K_0$ and $m_0>1$ such that
      $$\alpha_0|p|^{m_0}-K_0\le H(x,t,p)\le \beta_0|p|^{m_0}+K_0;$$
    \item[(P3)] for $(x,t)\in \mathbb R^{n+1}/\mathbb Z^{n+1}$, the map $p\mapsto H(x,t,p)$ is convex.
\end{enumerate}

We assume the initial data $g\in \buc(\mathbb R^n)\cap \lip(\mathbb R^n)$, where $\buc(\mathbb R^n)$ and $\lip(\mathbb R^n)$ are respectively the set of bounded uniformly continuous, and Lipschitz functions, on $\mathbb R^n$. By those assumptions, $u^\varepsilon$ converges locally uniformly on $\mathbb R^n\times [0,\infty)$ to a function $u$ as $\varepsilon\rightarrow 0$, and $u$ solves the effective equation
\begin{equation}\label{hj-eff-timedep}
    \begin{cases}
        u_t+\overline{H}\left(Du\right)&=0\quad\quad \ \ \text{in $\mathbb R^n\times (0,\infty)$},\\
        u(x,0)&=g(x)\quad \ \text{on $\mathbb R^n$}.
    \end{cases}
\end{equation}
with $\overline{H}$ is defined as follows. For each $p\in\mathbb R^n$, there is a unique constant $\overline{H}(p)$ such that
$$v_s+H(y,s,p+Dv)=\overline{H}(p)$$
has a continuous viscosity solution; see \cite{lpv, evans2, hung} for more details.

Condition (P2) is natural when the given Hamiltionian is time-dependent, see \cite{evans2}. Also, there are some positive constants $\alpha,\beta, K$ and $m>1$ satisfying
$$\alpha|v|^{m}-K\le L(x,t,v)\le \beta|v|^{m}+K$$
Here, $L(x,t,v)$ is the Legendre transform of $H(x,t,p)$ defined by
$$L(x,t,v)=\sup_{p\in \mathbb R^n}\left(p\cdot v-H(x,t,p)\right)$$
and $\frac{1}{m}+\frac{1}{m_0}=1$. 
% $H$ is coercive in $p$. Precisely, uniformly for $(x,t)\in \mathbb R^{n+1}/\mathbb Z^{n+1}$,
%         $$\lim_{|p|\rightarrow \infty}H(x,t,p)=\infty;$$
The optimal control formula for solution of (\ref{hj-cauchy-timedep}) is
    $$u^{\varepsilon}(x,t)=
        \inf_{\substack{\varepsilon\eta(0)=x \\ \eta\in \AC([-\varepsilon^{-1}t,0],\mathbb R^n)}}
        \left\{g\left(\varepsilon\eta \left(-\varepsilon^{-1}t\right)\right)+\varepsilon\int_{-\varepsilon^{-1}t}^{0}L\left(\eta(s),s,\dot{\eta}(s)ds\right) \right\}$$
where $\AC([-\varepsilon^{-1}t,0],\mathbb R^n)$ is the space of absolutely continuous curves in $\mathbb R^n$ defined on $[-\varepsilon^{-1}t,0]$.

\subsection{Previous literature}

We give a brief review on finding optimal convergence rate for homogenization of Hamilton-Jacobi equations.

\begin{enumerate}
    \item If one consider general Hamiltonian, the best known convergence rate is $O(\varepsilon^{1/3})$ by Capuzzo-Dolcetta and Ishii \cite{ishii}; see a detailed explanation at \cite{hung}. It only relies on classical methods, hence one needs further assumptions on Hamiltonian, such as convexity or Lipschitz property, to improve the result.
    \item When the Hamiltonian is convex, the lower bound $u^{\varepsilon}-u\ge C\varepsilon$ always holds by a result of Mitake, Tran and Yu \cite{near-optimal}, and a similar upper bound is also true for $n\le 2$. Besides, if we consider multi-scaled Hamiltonian in one-dimensional settings, Tu also showed that the convergence rate is $O(\varepsilon)$ \cite{son}.
    \item When $n\ge 3$, by using idea from first passage percolation method, Cooperman pointed out an unconditional upper bound $|u^\varepsilon(x,t)-u(x,t)|\le C\varepsilon\log(C+t\varepsilon^{-1})$, see \cite{cooperman}.
    \item Later on, Tran and Yu established the $O(\varepsilon)$ convergence rate \cite{hungyu}; their proof based on an unexpected lemma on periodic metrics by Burago, see \cite{burago}. It is likely that Burago's result could be adapted for more complicated situations, such as time-dependent or multi-scaled convex Hamiltonian.
\end{enumerate}

\subsection{Main results}

\begin{theorem}\label{hj-rate-timedep}
  Let $u^\varepsilon$ be the viscosity solution to \eqref{hj-cauchy-timedep} and $u$ be the viscosity solution to \eqref{hj-eff-timedep}. Then, there exists $C>0$ depending only on $H,\norm{Dg}_{L^\infty(\mathbb R^n)}$ and $n$ such that, for each $\varepsilon\in (0,1)$,
  $$\norm{u^\varepsilon-u}_{L^\infty(\mathbb R^n\times (0,\infty))}\le C\varepsilon.$$
\end{theorem}

Its proof generally follows that of \cite{hungyu}, but there are some new contributions. Since $L$ depends on time variable, one cannot apply exactly the same techniques used for time-independent case, which requires boundedness of velocity of the minimizing curves. Additionally, rescaling and periodic shifting processes are far more complicated. Periodic variables could only be shifted via integer vectors, hence if those previous ideas are reused, one needs to carefully take care of subcurves decomposed from the minimizers and connectors joining them.

\section{Proof of the main result}

Before starting, we restate an important lemma:

\begin{lemma}[Burago, \cite{burago}] Let $n\in \mathbb Z^+$ and $\xi:[0,1]\rightarrow \mathbb R^n$ be a continuous path. Then, there is a positive integer $k\le \frac{n+1}{2}$ and a collection of disjoint intervals $\{[a_i,b_i]\}_{1\le i\le k}$, which are subintervals of $[0,1]$, such that
    $$\sum_{i=1}^{k}\left(\xi(b_i)-\xi(a_i)\right)=\dfrac{\xi(1)-\xi(0)}{2}.$$
\end{lemma}

Note that Burago's lemma still holds for a path defined on $[a,b]$ by a simple rescaling; this observation will be useful later. Now for $x,y\in \mathbb R^n$ and $t>0$, denote by
$$m(t,x,y)=\inf\left\{\int_{0}^{t}L\left(\eta(s),s,\dot{\eta}(s)\right)ds \mid \eta\in\AC([0,t],\mathbb R^n),\eta(0)=x,\eta(t)=y\right\}.$$

One might imagine $m(t,x,y)$ as the minimum cost to travel from $x$ to $y$ in a given time $t>0$. The homogenized cost is
$$\overline{m}(t,x,y)=\lim_{k\rightarrow \infty}\dfrac{1}{k}m(kt,kx,ky).$$
Thanks to the Hopf-Lax formula for (\ref{hj-eff-timedep}), for any $t>0$,
$$\overline{m}(t,x,y)=t\overline{L}\left(\dfrac{y-x}{t}\right),$$
and also for any $s>0$,
$$\overline{m}(st,sx,sy)=st\overline{L}\left(\dfrac{s(y-x)}{st}\right)=s\overline{m}(t,x,y).$$

\begin{lemma}\label{hj-subadd-timedep}
    With all above assumptions, the following properties hold:
    \begin{enumerate}[label=\roman*)]
        % \item $m$ is subadditive: for $x,y,z\in \mathbb R^n$ and $t,s>0$,
        %     $$m(t,x,y)+m(s,y,z)\ge m(t+s,x,z).$$
        \item $m$ is $\mathbb Z^n-$periodic: for $x,y\in \mathbb R^n,w\in \mathbb Z^n$ and $t>0$,
            $$m(t,x+w,y+w)=m(t,x,y).$$
        \item For $t>0$ and $|y|\le Mt$:
            $$m(2t,0,2y)\le 2m(t,0,y)+C.$$
        In particular,
            $$\overline{m}(t,0,y)\le m(t,0,y)+C.$$
    \end{enumerate}
    Here, $C>0$ is a universal constant depending only on $K,M,n,\alpha$ and $\beta$. 
\end{lemma}

\begin{proof}
    \begin{enumerate}[label=\roman*)]
        % \item Denote minimizers of $m(t,x,y)$ and $m(s,y,z)$ respectively by $\eta_1\in \AC([0,t],\mathbb R^n)$ and $\eta_2\in \AC([0,s],\mathbb R^n)$. We construct a path $\eta\in \AC([0,t+s],\mathbb R^n)$ as follow:
        % $$\eta(v)=\begin{cases}
        %     \eta_1(v);& 0\le v<t\\
        %     \eta_2(v-t);& t\le v\le t+s
        % \end{cases}$$
        % Now use the definition of $m(t+s,x,z)$ to obtain
        % \begin{equation*}
        %     \begin{split}
        %         m(t,x,y)+m(s,y,z)&=\int_{0}^{t}L\left(\eta_1(v),v,\dot{\eta_1}(v)\right)dv+\int_{0}^{s}L\left(\eta_2(v),v,\dot{\eta_2}(v)\right)dv\\
        %         &=\int_{0}^{t}L\left(\eta(v),v,\dot{\eta}(v)\right)dv+\int_{t}^{t+s}L\left(\eta(v),v,\dot{\eta}(v)\right)dv
        %     \end{split}
        % \end{equation*}

        \item A minimizer $\eta_1$ of $m(t,x,y)$ induces a path $\eta_2$ from $x+w$ to $y+w$ by shifting. From the $\mathbb Z^{n+1}-$periodicity of $L(y,t,q)$, we deduce:
        \begin{equation*}
            \begin{split}
                m(t,x,y)&=\int_{0}^{t}L\left(\eta_1(s),s,\dot{\eta_1}(s)\right)ds=\int_{0}^{t}L\left(\eta_1(s)+w,s,\dot{\eta_1}(s)\right)ds\\
                &=\int_{0}^{t}L\left(\eta_2(s),s,\dot{\eta_2}(s)\right)ds\ge m(t,x+w,y+w).
            \end{split}
        \end{equation*}
        The reverse inequality is obtained by shifting backward.
        \item For each $t>0$ and $y\in \mathbb R^n$ with $|y|\le Mt$, consider $\eta\in \AC([0,t],\mathbb R^n)$ joining $0$ and $y$ with constant velocity. From the polynomial growth rate of $L(\ .\ ,q)$,
            $$m(t,0,y)\le \int_{0}^{t}\left(\beta \cdot \dfrac{|y|^{m}}{t^{m}}+K\right)dv \le \beta\cdot\dfrac{(Mt)^{m}}{t^{m-1}}+Kt\le Ct.$$
        Similarly $-Ct\le m(t,0,y)$. Hence the theorem is true if $t\le 6$, since we will get
        $$m(2t,0,2y)-2m(t,0,y)\le 2Ct+2Ct\le C.$$
        Now only consider $t>6$. Let $k$ be the greatest positive integer such that $6k\le t$. Also, if we denote the minimizer of $m(t,0,y)$ by $\eta$, then
        \begin{equation*}
            \begin{split}
            \sum_{l=0}^{k-1}\int_{l}^{l+6}L\left(\eta(s),s,\dot{\eta}(s)\right)ds &= \int_{0}^{t}L\left(\eta(s),s,\dot{\eta}(s)\right)ds - \int_{6k}^{t}L\left(\eta(s),s,\dot{\eta}(s)\right)ds\\
            &\le m(t,0,y)-\int_{6k}^{t}\left(\alpha\left|\dot{\eta}(s)\right|^{m}-K\right)ds\\
            &\le Ct+(t-6k)K\le Ct+C\le Ct.
            \end{split}
        \end{equation*}
        Then, there exists $l\in\{0,1,\cdots,k-1\}$ satisfying
            $$\int_{l}^{l+6}L\left(\eta(s),s,\dot{\eta}(s)\right)ds \le \dfrac{Ct}{k}\le C.$$
        From our assumption on $L(\ .\ ,v)$, we obtain
            $$\int_{l}^{l+6}\left(\alpha\left|\dot{\eta}(s)\right|^{m}-K\right)ds\le \int_{l}^{l+6}L\left(\eta(s),s,\dot{\eta}(s)\right)ds \le C,$$
        which can be rewritten as
            $$\int_{l}^{l+6}\left|\dot{\eta}(s)\right|^{m}ds\le C.$$
        The idea here is straightforward: we ``double" the trace of $m(t,0,y)$ and shift it in an approriate way such that their endpoints are as close as possible. Hence one should ``move faster" on some part of the initial curve such that its velocity is bounded, to save time for connectors. See the picture below.

        \begin{center}
            \includegraphics[scale=0.6]{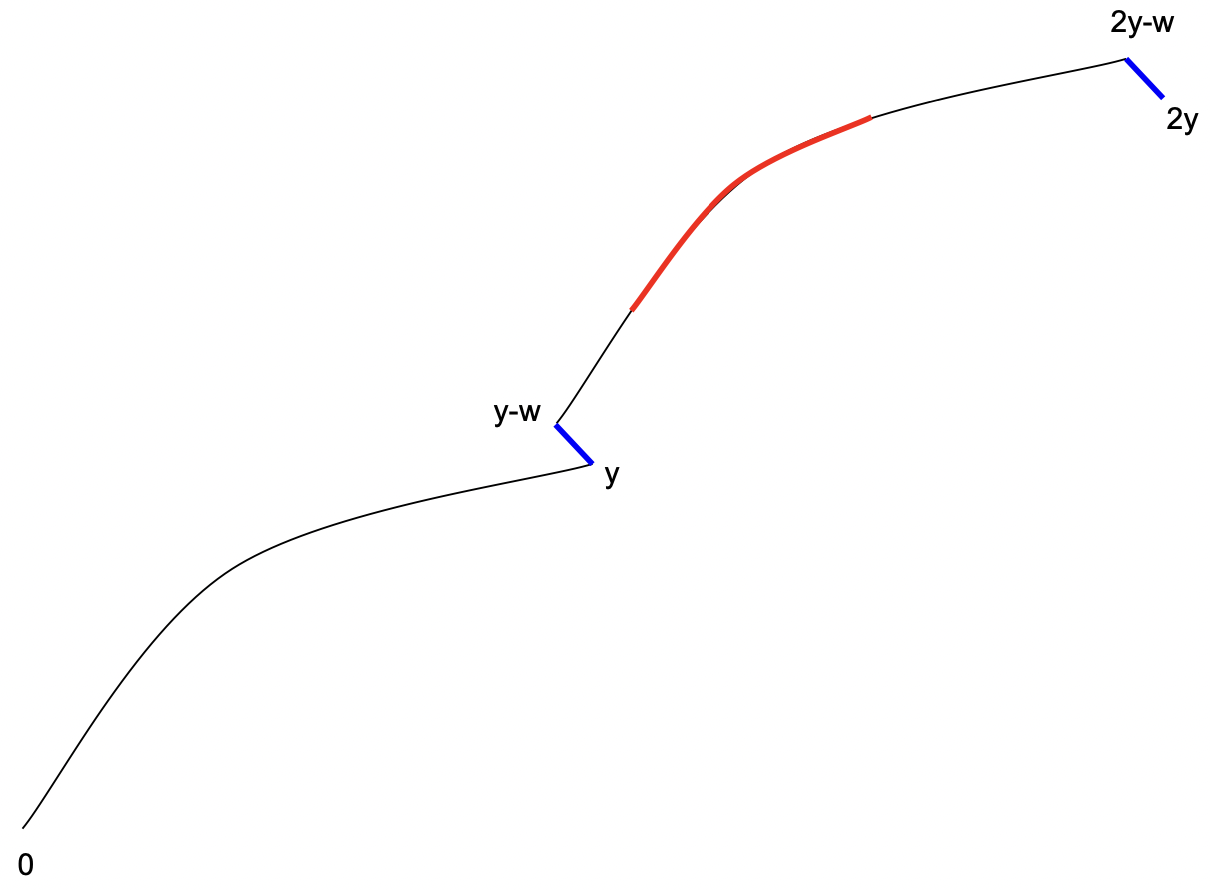}
        \end{center}
    
        We write down this idea rigourously. Let $w\in[0,1]^n$ such that $y-w\in \mathbb Z^n$. Construct a path $\mu\in\AC([0,2t],\mathbb R^n)$ with $\mu(0)=0,\mu(2t)=2y$ as below:
        $$\mu(s)=\begin{cases}
            \eta(s),& 0\le s<t;\\
            \dfrac{-s+t}{\lceil t \rceil+2-t}w+y, &t\le s\le \lceil t\rceil+2;\\
            \eta(s-\lceil t\rceil-2)+y-w, &\lceil t\rceil +2\le s<\lceil t\rceil +l+2;\\
            \eta\left(6\left(s-\lceil t\rceil-l-2\right)+l\right)+y-w, &\lceil t\rceil +l+2\le s<\lceil t\rceil +l+3;\\
            \eta(s-\lceil t\rceil+3)+y-w, &\lceil t\rceil +l+3\le s<\lceil t\rceil +t-3;\\
            \dfrac{s-\lceil t\rceil-t+3}{t-\lceil t\rceil+3}w+2y-w; &\lceil t\rceil +t-3\le s\le 2t.   
        \end{cases}$$
        For $1\le i\le 6$, denote the cost of $\mu$ on $i-$th time interval by $K_i$. Note that
        $$m(2t,0,2y)\le \int_{0}^{2t}L\left(\mu(s),s,\dot{\mu}(s)\right)ds=\sum_{i=1}^{6}K_i,$$
        so the rest is to estimate each $K_i$ to show that their sum is close to $2m(t,0,y)$. Indeed:
        \begin{itemize}
            \item Estimations for $K_1$:
            \begin{equation*}
                \begin{split}
                K_1 &= \int_{0}^{t}L\left(\mu(s),s,\dot{\mu}(s)\right)ds = \int_{0}^{t}L\left(\eta_1(s),s,\dot{\eta_1}(s)\right)ds = m(t,0,y).
                \end{split}
            \end{equation*}
            \item Estimations for $K_2$:
            \begin{equation*}
                \begin{split}
                K_2 &= \int_{t}^{\lceil t\rceil+2}L\left(\mu(s),s,\dot{\mu}(s)\right)ds\le \int_{t}^{\lceil t\rceil+2}\left(\beta\left|\dot{\mu}(s)\right|^m+K\right)ds\\
                &\le \int_{t}^{\lceil t\rceil+2}\left(\beta\cdot\dfrac{1}{\left(\lceil t\rceil+2-t\right)^m}+K\right)ds = \dfrac{\beta}{\left(\lceil t\rceil+2-t\right)^{m-1}}+K\left(\lceil t\rceil+2-t\right)\\
                &\le \dfrac{\beta}{2^m}+3K\le C.
                \end{split}
            \end{equation*}
            \item Estimations for $K_6$:
            \begin{equation*}
                \begin{split}
                K_6 &= \int_{\lceil t\rceil +t-3}^{2t}L\left(\mu(s),s,\dot{\mu}(s)\right)ds\le \int_{\lceil t\rceil +t-3}^{2t}\left(\beta\left|\dot{\mu}(s)\right|^m+K\right)ds\\
                &\le \int_{t}^{\lceil t\rceil+2}\left(\beta\cdot\dfrac{1}{\left(t-\lceil t\rceil+3\right)^m}+K\right)ds = \dfrac{\beta}{\left(t-\lceil t\rceil+3\right)^{m-1}}+K\left(t-\lceil t\rceil+3\right)\\
                &\le \dfrac{\beta}{2^m}+3K\le C.
                \end{split}
            \end{equation*}
            \item Estimations for $K_4$:
            \begin{equation*}
                \begin{split}
                K_4 &= \int_{\lceil t\rceil +l+2}^{\lceil t\rceil +l+3}L\left(\mu(s),s,\dot{\mu}(s)\right)ds\\
                &= \int_{\lceil t\rceil +l+2}^{\lceil t\rceil +l+3}L\bigl(\eta\left(6\left(s-\lceil t\rceil-l-2\right)+l\right)+y-w, s, \dot{\eta}\left(6\left(s-\lceil t\rceil-l-2\right)+l\right)\bigr)ds\\
                &=\dfrac{1}{6}\int_{l}^{l+6}L\left(\eta(s),\dfrac{s-l}{6}+\lceil t\rceil+l+2,6\dot{\eta}(s)\right)ds\\
                &\le \dfrac{1}{6}\int_{l}^{l+6}\left(\beta\cdot\left|6\dot{\eta}(s)\right|^m+K\right)ds\\
                &\le \dfrac{\beta}{6^{m-1}}\int_{l}^{l+6}\left|\dot{\eta}(s)\right|^mds+K\le C.
                \end{split}
            \end{equation*}
            \item Estimation for $K_3$:
            \begin{equation*}
                \begin{split}
                K_3 &= \int_{\lceil t\rceil +2}^{\lceil t\rceil +l+2}L\left(\mu(s),s,\dot{\mu}(s)\right)ds\\
                &= \int_{\lceil t\rceil +2}^{\lceil t\rceil +l+2}L\left(\eta(s-\lceil t\rceil-2)+y-w,s,\dot{\eta}(s-\lceil t\rceil-2)\right)ds\\
                &= \int_{0}^{l}L\left(\eta(s)+y-w,s+\lceil t\rceil+2,\dot{\eta}(s)\right)ds=\int_{0}^{l}L\left(\eta(s),s,\dot{\eta}(s)\right)ds.
                \end{split}
            \end{equation*}
            \item Estimation for $K_5$:
            \begin{equation*}
                \begin{split}
                K_5 &= \int_{\lceil t\rceil +l+3}^{\lceil t\rceil +t-3}L\left(\mu(s),s,\dot{\mu}(s)\right)ds\\
                &= \int_{\lceil t\rceil +l+3}^{\lceil t\rceil +t-3}L\left(\eta(s-\lceil t\rceil+3)+y-w,s,\dot{\eta}(s-\lceil t\rceil+3)\right)ds\\
                &= \int_{l+6}^{t}L\left(\eta(s)+y-w,s+\lceil t\rceil-3,\dot{\eta}(s)\right)ds=\int_{l+6}^{t}L\left(\eta(s),s,\dot{\eta}(s)\right)ds.
                \end{split}
            \end{equation*}
        \end{itemize}
        We combine all above results to yield
            \begin{equation*}
                \begin{split}
                    m(2t,0,2y)&\le \sum_{i=1}^{6}K_i\\
                    &\le m(t,0,y)+\int_{0}^{l}L\left(\eta(s),s,\dot{\eta}(s)\right)ds+\int_{l+6}^{t}L\left(\eta(s),s,\dot{\eta}(s)\right)ds+C\\
                    &= 2m(t,0,y)+C-\int_{l}^{l+6}L\left(\eta(s),s,\dot{\eta}(s)\right)ds\\
                    &\le 2m(t,0,y)+C-\int_{l}^{l+6}\left(\alpha\left|\dot{\eta}(s)\right|^m-K\right)ds\\
                    &\le 2m(t,0,y)+C+6K\le 2m(t,0,y)+C.
                \end{split}
            \end{equation*}
        Now we repeatly use this inequality to obtain
        $$m(2^kt,0,2^ky)\le 2^km(t,0,y)+C(2^k-1),$$
        which could be rewritten as
        $$\dfrac{1}{2^k}m(2^kt,0,2^ky)\le m(t,0,y)+C\left(1-\dfrac{1}{2^k}\right)$$
        for every positive integer $k$. Now let $k$ tend to infinity to finish the proof.
    \end{enumerate}
\end{proof}

Next, we prove the superadditive property.
\begin{lemma}
    Given any constant $M>0$, if $|y|\le Mt$ then
        $$2m(t,0,y)\le m(2t,0,2y)+C.$$
    In particular,
        $$m(t,0,y)\le \overline{m}(t,0,y)+C.$$
    Here, $C>0$ is a universal constant depending only on $K,M,n,\alpha$ and $\beta$.
\end{lemma}
\begin{proof}
    By similar arguments in the proof of Lemma \ref{hj-subadd-timedep}, we get
    $$-Ct\le m(t,0,y),m(2t,0,2y)\le Ct.$$
    Hence if $t<4(n+4)^2$, the superadditive property is obvious since
    $$2m(t,0,y)-m(2t,0,2y)\le 2Ct+Ct\le C.$$
    Now only consider $t>4(n+4)^2$. Let $\eta$ be a minimizer for $m(2t,0,2y)$. Note that $\eta\in\AC([0,2t],\mathbb R^n)$, also $\eta(0)=0$ and $\eta(2t)=2y$. By using Burago's lemma for continuous curve $\zeta:[0,2t]\rightarrow \mathbb R^{n+1}$ satisfying $\zeta(s)=(\eta(s),s)$, there exist a positive integer $k\le \frac{n+2}{2}$ and a collection of disjoint subintervals $\{[a_i,b_i]\}_{1\le i\le k}$ of $[0,2t]$ such that
    $$\sum_{i=1}^{k}\left(\eta(b_i)-\eta(a_i)\right)=\dfrac{\eta(2y)-\eta(0)}{2}=\dfrac{2y}{2}=y,\quad \text{and }\sum_{i=1}^{k}(b_i-a_i)=\dfrac{2t-0}{2}=t.$$

    For convenience, let $d_0=0$. Inductively define $\{c_i\}$ and $\{d_i\}$ for $1\le i\le k$ as follow:
    \begin{itemize}
        \item for each $1\le i\le k$, let $c_i$ be the smallest number such that $c_i-a_{i}\in \mathbb Z$ and $c_i-d_{i-1}\ge 1$; note that from this definition, $1\le c_i-d_{i-1}<2$;
        \item after defining $c_i$, choose $d_i$ satisfying $d_i-c_i=b_i-a_i$.
    \end{itemize}

    \begin{center}
        \includegraphics[scale=0.6]{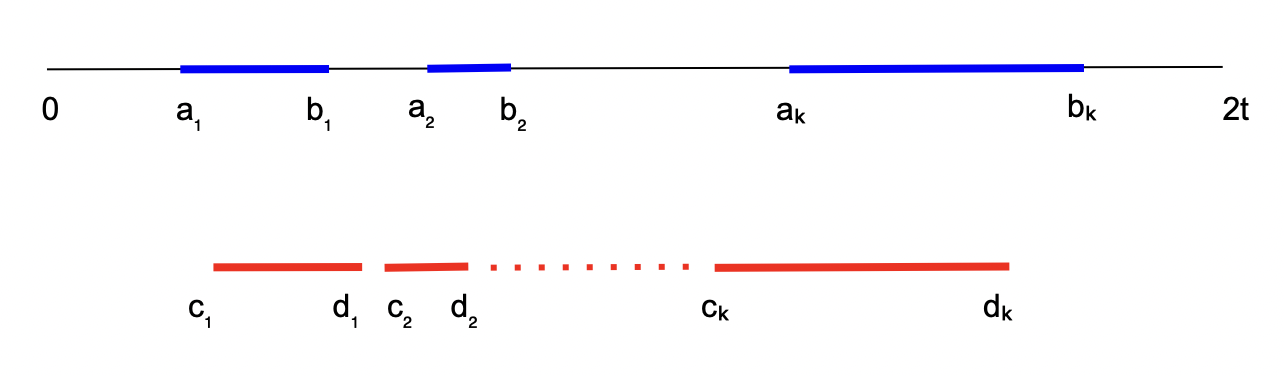}
    \end{center}

    Those difference will be time-difference for connectors between corresponding curve segments. Moreover, their time endpoints should be far enough so that the velocity of connectors would not blow up. Now on each interval $[c_i,d_i]$, we proceed periodic shifting by defining paths $\mu_i\in\AC([c_i,d_i],\mathbb R^n)$ satisfying:
    \begin{itemize}
        \item $\mu_1(c_i)\in [0,1]^n$;
        \item there exists $w_i\in \mathbb Z^n$ such that
            $$\mu_i(s+c_i-a_i)-\eta(s)=w_i\quad \text{for each $s\in[a_i,b_i]$},$$
        which leads to 
            $$\sum_{i=1}^{k}(\mu_i(d_i)-\mu_i(c_i))=\sum_{i=1}^{k}(\eta(b_i)-\eta(a_i))=y;$$
        \item for $1\le i\le k-1$, $\mu_{i}(d_i)$ and $\mu_{i+1}(c_{i+1})$ land in the same unit cube, so 
            $$\left|\mu_{i}(d_i)-\mu_{i+1}(c_{i+1})\right|\le \sqrt{n}.$$
        In other words, the spatial endpoints are as close as possible.
    \end{itemize}

    From the above definition and the periodicity of $L$, we have
    \begin{equation*}
        \begin{split}
            \int_{c_i}^{d_i}L(\mu_i(s),s,\dot{\mu_i}(s))ds &= \int_{c_i}^{d_i}L(\mu_i(s),s+a_i-c_i,\dot{\mu_i}(s))ds\\
            &=\int_{c_i}^{d_i}L(\eta(s+a_i-c_i),s+a_i-c_i,\dot{\eta}(s+a_i-c_i))ds\\
            &=\int_{a_i}^{b_i}L(\eta(s),s,\dot{\eta}(s)).ds
        \end{split}
    \end{equation*}

    If we define $M$ as a least integer such that
    $$M>\sum_{i=0}^{k-1}(c_{i+1}-d_i),$$
    then
    $$M\le \sum_{i=0}^{k-1}(c_{i+1}-d_i)+1\le 2k+1\le n+3\le 4n^2<t.$$
    Now we point out an index $1\le j\le k$ and a subinterval $[l,l+3M]$ of $[c_j,d_j]$ satisfying
    $$\int_{l}^{l+3M}L(\mu_j(s),s,\dot{\mu_j}(s))ds\le C$$
    to reuse our time-saving idea. First, there exists $j$ such that $d_j-c_j>3M$ and
    $$\int_{c_j}^{d_j}L(\mu_j(s),s,\dot{\mu_j}(s))ds\le Ct.$$
    Indeed, the existence of $j$ follows from
    \begin{equation*}
        \begin{split}
            \max_{1\le i\le k}(d_i-c_i) &\ge \dfrac{1}{k}\sum_{i=1}^{k}(d_i-c_i)=\dfrac{1}{k}\sum_{i=1}^{k}(b_i-a_i)=\dfrac{t}{k}\\
            &\ge \dfrac{2t}{n+2}\ge \dfrac{8(n+4)^2}{n+2}>3(n+3)\ge 3M.
        \end{split}
    \end{equation*}
    And by defining $b_0=0$ and $a_{k+1}=2t$, we obtain
    \begin{equation*}
        \begin{split}
            & \int_{c_j}^{d_j}L(\mu_i(s),s,\dot{\mu_i}(s))ds\\
            & = \int_{a_j}^{b_j}L(\eta(s),s,\dot{\eta}(s))ds\\
            &=\int_{0}^{2t}L(\eta(s),s,\dot{\eta}(s))ds-\int_{0}^{a_j}L(\eta(s),s,\dot{\eta}(s))ds-\int_{b_j}^{2t}L(\eta(s),s,\dot{\eta}(s))ds\\
            &\le m(2t,0,2y)-\int_{0}^{a_j}(\alpha\left|\dot{\eta}(s)\right|^m-K)ds-\int_{b_j}^{2t}(\alpha\left|\dot{\eta}(s)\right|^m-K)ds\\
            &\le m(2t,0,2y)+K(a_j+2t-b_j)\le Ct+2Kt\le Ct.
        \end{split}
    \end{equation*}
    Again by subdividing $[c_i,d_i]$ into consecutive segments with length $3M$ and possibly a redundant part shorter than $3M$, then using arguments in the proof of Lemma \ref{hj-subadd-timedep}, there exists a subinterval $[l,l+3M]$ of $[c_j,d_j]$ satisfying
    $$\int_{l}^{l+3M}L(\mu_i(s),s,\dot{\mu_i}(s))ds\le Ct\cdot\dfrac{3M}{t}\le C.$$
    From this we easily deduce
    $$-C\le \int_{l}^{l+3M}\left|\dot{\eta}(s)\right|^mds\le C.$$

    That boundedness saves a time interval length $2M$ for straight connectors. Now the essential step is to construct a path $\zeta\in \AC([0,t],\mathbb R^n)$ such that
    $$\int_{0}^{t}L(\zeta(s),s,\dot{\zeta}(s))ds\le \sum_{i=1}^{k}\int_{a_i}^{b_i}L(\eta(s),s,\dot{\eta}(s))ds+C.$$
    We define $\zeta$ as follows:
    \begin{itemize}
        \item for $0\le s< c_1$,
            $$\zeta(s)=\dfrac{\mu_1(c_1)s}{c_1};$$
        \item if $1\le i<j$, then
            $$\zeta(s)=\begin{cases}
                \mu_i(s),& c_i\le s<d_i,\\
                \dfrac{(\mu_{i+1}(c_{i+1})-\mu_i(d_i))(s-d_i)}{c_{i+1}-d_i}+\mu_i(d_i),& d_i\le s<c_{i+1};
            \end{cases}$$
        \item if $i=j$, then
            $$\zeta(s)=\begin{cases}
                \mu_j(s), & c_i\le s<l,\\
                \mu_j(3s-2l), & l\le s<l+M,\\
                \mu_j(s+2M), & l+M\le s<d_j-2M,\\
                \dfrac{(\mu_{j+1}(c_{j+1})-\mu_j(d_j))(s+2M-d_j)}{c_{j+1}-d_j}+\mu_j(d_j), & d_j-2M\le s<c_{j+1}-2M;
            \end{cases}$$
        \item if $j<i<k$, then
            $$\zeta(s)=\begin{cases}
                \mu_i(s+2M),& c_i-2M\le s<d_i-2M,\\
                \dfrac{(\mu_{i+1}(c_{i+1})-\mu_i(d_i))(s+2M-d_i)}{c_{i+1}-d_i}+\mu_i(d_i),& d_i-2M\le s<c_{i+1}-2M;
            \end{cases}$$
        \item if $i=k$, then
            $$\zeta(s)=\begin{cases}
                \mu_i(s+2M),& c_k-2M\le s<d_k-2M,\\
            \dfrac{(y-\mu_k(d_k))(s-d_k+2M)}{t-d_k
            +2M}+\mu_k(d_k),& d_k-2M\le s\le t.
            \end{cases}$$
    \end{itemize}

    \begin{center}
        \includegraphics[scale=0.6]{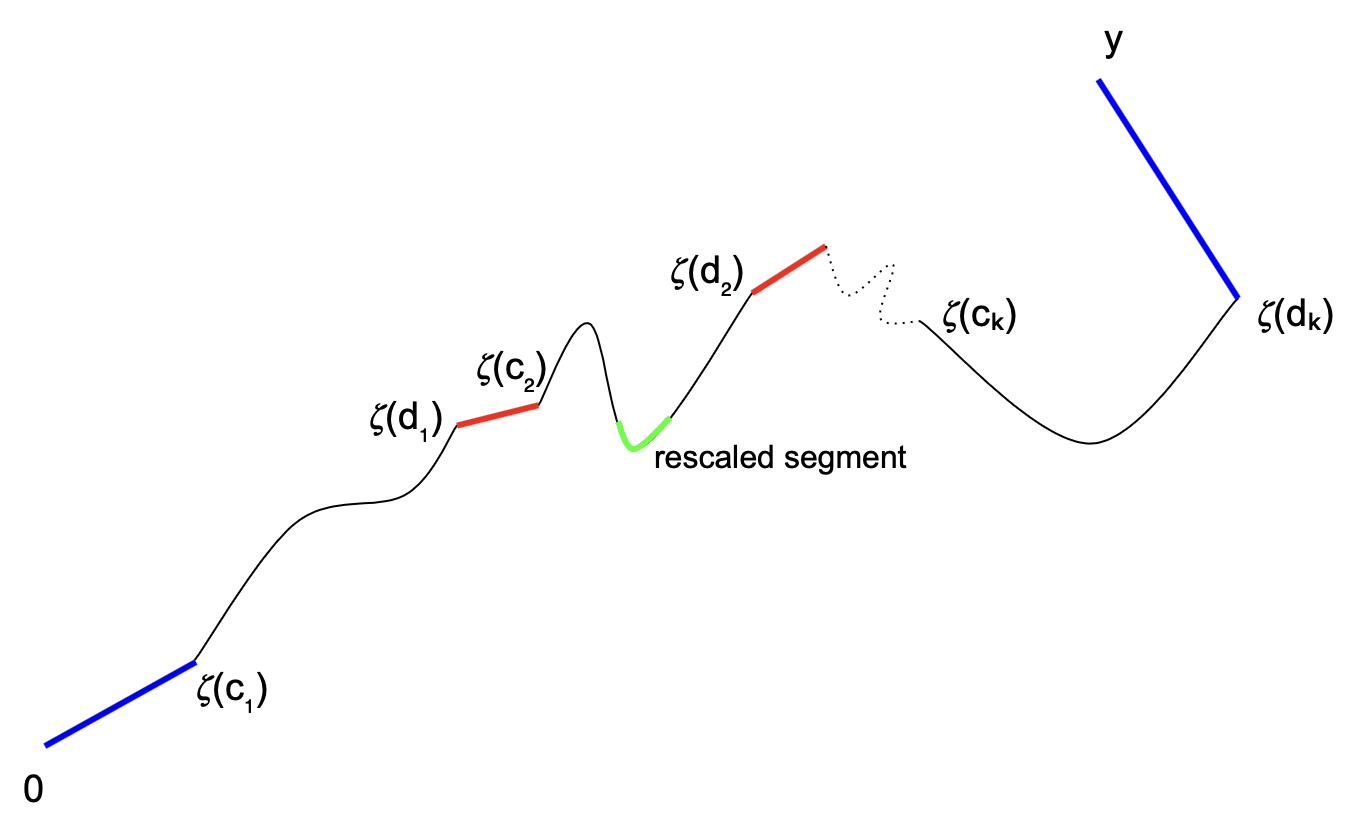}
    \end{center}

    See the above picture. The rest is just calculation.
    \begin{itemize}
        \item Keep in mind that $d_0=0$, then for $0\le i<k$:
        \begin{equation*}
            \begin{split}
                \int_{d_i}^{c_{i+1}}L(\zeta(s),s,\dot{\zeta}(s))ds &\le \int_{d_i}^{c_{i+1}}\left(\beta\left|\dot{\zeta}(s)\right|^m+K\right)ds\\
                &\le \beta\int_{d_i}^{c_{i+1}}\left|\dot{\zeta}(s)\right|^mds+K(c_{i+1}-d_i)\\
                &\le \beta\int_{d_i}^{c_{i+1}}\max_{1\le i<k}\left\{\left|\dfrac{\mu_1(c_1)}{c_1}\right|^m,\left|\dfrac{\mu_{i+1}(c_{i+1})-\mu_i(d_i)}{c_{i+1}-d_i}\right|^m\right\}ds+2K\\
                &\le \beta(c_{i+1}-d_i)\left(\dfrac{\sqrt{n}}{2}\right)^m+C\le 2\beta\left(\dfrac{\sqrt{n}}{2}\right)^m+C\le C.
            \end{split}
        \end{equation*}
        \item The cost of $\zeta$ on $[d_k-2M,t]$ is bounded by a constant. Indeed:
        \begin{equation*}
            \begin{split}
                \int_{d_k-2M}^{t}L(\zeta(s),s,\dot{\zeta}(s))ds &\le \int_{d_k-2M}^{t}\left(\beta\left|\dot{\zeta}(s)\right|^m+K\right)ds\\
                &\le \beta\int_{d_k-2M}^{t}\dfrac{\left|y-\mu_k(d_k)\right|^m}{(t-d_k+2M)^m}ds+KM\\
                &\le \dfrac{\beta}{(t-d_k+2M)^m}\int_{t-M}^{t}{\left|\sum_{i=1}^{k}\left(\mu_i(d_i)-\mu_i(c_i)\right)-\mu_k(d_k)\right|^m}ds+C\\
                &\le C\int_{t-M}^{t}\left(\sum_{i=0}^{k-1}\left|\mu_i(d_i)-\mu_{i+1}(c_{i+1})\right|\right)^mds+C\\
                &\le C\int_{t-M}^{t}\left(k\sqrt{n}\right)^mds+C=Ck^m\left(\sqrt{n}\right)^m+C\le C.
            \end{split}
        \end{equation*}
        \item When $1\le i<j$ and $s$ lies outside $[c_i,d_i]$, we have
            $$\int_{c_i}^{d_i}L(\zeta(s),s,\dot{\zeta}(s))ds = \int_{c_i}^{d_i}L(\mu_i(s),s,\dot{\mu_i}(s))ds = \int_{a_i}^{b_i}L(\eta(s),s,\dot{\eta}(s))ds.$$
        \item When $j<i\le k$ and $s$ lies in $[c_i-2M,d_i-2M]$, we have
        \begin{equation*}
            \begin{split}
                \int_{c_i-2M}^{d_i-2M}L(\zeta(s),s,\dot{\zeta}(s))ds &= \int_{c_i-2M}^{d_i-2M}L(\mu_i(s+2M),s,\dot{\mu_i}(s+2M))ds\\
                &= \int_{c_i}^{d_i}L(\mu_i(s),s-2M,\dot{\mu_i}(s))ds\\
                &= \int_{c_i}^{d_i}L(\mu_i(s),s,\dot{\mu_i}(s))ds = \int_{a_i}^{b_i}L(\eta(s),s,\dot{\eta}(s))ds.
            \end{split}
        \end{equation*}
        \item When $i=j$, we obtain,
            $$\int_{c_i}^{l}L(\zeta(s),s,\dot{\zeta}(s))ds = \int_{c_i}^{l}L(\mu_i(s),s,\dot{\mu_i}(s))ds$$
        and also
        \begin{equation*}
            \begin{split}
                \int_{l+M}^{d_j-2M}L(\zeta(s),s,\dot{\zeta}(s))ds &= \int_{l+M}^{d_j-2M}L(\mu_i(s+2M),s,\dot{\mu_i}(s+2M))ds\\
                &= \int_{l+3M}^{d_j}L(\mu_i(s),s-2M,\dot{\mu_i}(s))ds\\
                &= \int_{l+3M}^{d_j}L(\mu_i(s),s,\dot{\mu_i}(s))ds.
            \end{split}
        \end{equation*}
        Hence,
        \begin{equation*}
            \begin{split}
                & \int_{c_i}^{l}L(\zeta(s),s,\dot{\zeta}(s))ds+\int_{l+M}^{d_j-2M}L(\zeta(s),s,\dot{\zeta}(s))ds\\ &= \int_{c_i}^{d_j}L(\mu_i(s),s,\dot{\mu_i}(s))ds-\int_{l}^{l+3M}L(\mu_i(s),s,\dot{\mu_i}(s))ds\\
                &\le \int_{c_i}^{d_j}L(\mu_i(s),s,\dot{\mu_i}(s))ds +C.
            \end{split}
        \end{equation*}
        Moreover, since velocity of $\zeta$ is bounded on $[l,l+3M]$:
        \begin{equation*}
            \begin{split}
                \int_{l}^{l+M}L(\zeta(s),s,\dot{\zeta}(s))ds &= \int_{l}^{l+M}L(\mu_j(3s-2l),s,3\dot{\mu_j}(3s-2l))ds\\
                &= \dfrac{1}{3}\int_{l}^{l+3M}L\left(\mu_j(s),\dfrac{s+2l}{3},3\dot{\mu_j}(s)\right)ds\\
                &\le \dfrac{1}{3}\int_{l}^{l+3M}\left(\beta\left|\dot{\zeta}(s)\right|^m+K\right)ds\\
                &= \dfrac{\beta}{3}\int_{l}^{l+3M}\left|\dot{\zeta}(s)\right|^mds+KM\le C.
            \end{split}
        \end{equation*}

    \end{itemize}
    Therefore by decomposing $\zeta$ and summing by parts,
    $$\int_{0}^{t}L(\zeta(s),s,\dot{\zeta}(s))ds\le \sum_{i=1}^{k}\int_{a_i}^{b_i}L(\eta(s),s,\dot{\eta}(s))ds+C$$
    as desired. From here, we conclude
    \begin{equation}\label{super1}
        m(t,0,y)\le \sum_{i=1}^{k}\int_{a_i}^{b_i}L(\eta(s),s,\dot{\eta}(s))ds+C.
    \end{equation}
    
    Replicating all previous arguments for a collection of $k+1$ disjoint subintervals $\{[u_i,v_i]\}_{1\le i\le k+1}$ of $[0,2t]$, which is defined by
    $$[u_i,v_i]=\begin{cases}
        [0,a_i], &i=1,\\
        [b_{i-1},a_i], &2\le i\le k,\\
        [b_k,2t],&i=k+1,
    \end{cases}$$
    and also satisfies
    $$\sum_{i=1}^{k}\left(\eta(v_i)-\eta(u_i)\right)=\dfrac{\eta(2y)-\eta(0)}{2}=\dfrac{2y}{2}=y,\quad \text{and }\sum_{i=1}^{k}(v_i-u_i)=\dfrac{2t-0}{2}=t,$$
    the below estimation also holds:
    \begin{equation}\label{super2}
        m(t,0,y)\le \sum_{i=1}^{k+1}\int_{u_i}^{v_i}L(\eta(s),s,\dot{\eta}(s))ds+C.
    \end{equation}

    Adding (\ref{super1}) and (\ref{super2}) to finish the proof.
\end{proof}

Now we finalize the proof of Theorem \ref{hj-rate-timedep}.

\begin{proof}[Proof of Theorem 1.1]
    Thanks to Lemma 2.2.ii) and Lemma 2.3, for $t>0$ and $y\in \mathbb R^n$ such that $|y|\le Ct$,
    $$|m(t,0,y)-\overline{m}(t,0,y)|\le Ct.$$
    
    By scaling and translation, it suffices to prove the result for $(x,t)=(0,1)$. Since $g\in \lip(\mathbb R^n)$,
    $$|g(x)|\le |g(x)-g(0)|+|g(0)|\le C(|x|+1)$$
    for some constant $C$. Recall the optimal control formula:
    $$u^{\varepsilon}(0,1)=
    \inf_{\substack{\eta(0)=0 \\ \eta\in \AC([-\varepsilon^{-1},0],\mathbb R^n)}}
    \left\{g\left(\varepsilon\eta \left(-\varepsilon^{-1}\right)\right)+\varepsilon\int_{-\varepsilon^{-1}}^{0}L\left(\eta(s),s,\dot{\eta}(s)\right)ds \right\}.$$
    From the boundedness of $L$ and Jensen's inequality,
    \begin{equation*}
        \begin{split}
            \varepsilon\int_{-\varepsilon^{-1}}^{0}L\left(\eta(s),s,\dot{\eta}(s)ds\right) &\ge \varepsilon\int_{-\varepsilon^{-1}}^{0}\left(\alpha\left|\dot{\mu}(s)\right|^m-K\right)ds\\
            &\ge \alpha\left|\varepsilon\int_{-\varepsilon^{-1}}^{0}\dot{\eta}(s)ds\right|^m-K\\
            &= \alpha\varepsilon^m\left|\eta(-\varepsilon^{-1})\right|^m-K.
        \end{split}
    \end{equation*}

    Thus, for some constant $C>0$ depending only on $L,\norm{Dg}_{L^\infty(\mathbb R^n)}$ and $n$,
    \begin{equation*}
        \begin{split}
            u^{\varepsilon}(0,1) &= \inf_{\substack{\eta(0)=0 \\ \varepsilon\left|\eta(-\varepsilon^{-1})\right|\le C}}\left\{g\left(\varepsilon\eta \left(-\varepsilon^{-1}\right)\right)+\varepsilon\int_{-\varepsilon^{-1}}^{0}L\left(\eta(s),s,\dot{\eta}(s)\right)ds \right\}\\
            &= \inf_{|y|\le C}\left(g(y)+\varepsilon m(\varepsilon^{-1},-\varepsilon^{-1}y,0)\right)\\
            &= \inf_{|y|\le C}\left(g(y)+\overline{m}(1,0,-y)\right)+O(\varepsilon)\\
            &= u(0,1)+O(\varepsilon).
        \end{split}
    \end{equation*}
    Hence the proof is completed.
\end{proof}

% \section{Remarks}

% Assumption (P2) might be weakened to
% $$\alpha_0|p|^{m_1}-K_0\le H(x,t,p)\le \beta_0|p|^{m_2}+K_0$$
% for some positive constants $1<m_1<m_2$. However, the proof would be notably more complicated because if we follow lines by lines the former procedure, we could not derive the boundedness of velocity in $L^{m_2}(\mathbb R)$ of a subcurve. To overcome this, one should find a long enough time interval, which means its length is bounded from below by some constant $C_0$, such that the corresponding velocity is also bounded in $L^{1}(\mathbb R)$. Indeed, if that condition is satisfied, then
% $$\int_{a}^{b}\left|\dot{\eta}(s)\right|^{m_2}\le C_0^{m_2-m_1}\int_{a}^{b}\left|\dot{\eta}(s)\right|^{m_1}\le C$$
% and the rest is straightforward. This could be done by taking integral of level sets.


\bibliographystyle{amsplain}
\bibliography{hj}

\end{document}